\documentclass[12pt]{llncs}
\usepackage[letterpaper,hmargin=1in,vmargin=1.25in]{geometry}
\usepackage{amsmath}
\usepackage{amssymb}
\usepackage{amsfonts}
\usepackage{graphicx}
\usepackage{multirow}

\begin{document}

\title{Some Conjectures on the Number of Primes\\in Certain Intervals}

\author{Adway Mitra\inst{1},Goutam Paul\inst{2},Ushnish Sarkar\inst{3}}
\institute{Department of Computer Science and Automation,\\
Indian Institute of Science, Bangalore 560 012, India.\\
\email{adway@csa.iisc.ernet.in}
\and
Department of Computer Science and Engineering,\\
Jadavpur University, Kolkata 700 032, India.\\
\email{goutam\_paul@cse.jdvu.ac.in}
\and
Department of Electronics and Communication Engineering,\\
St. Thomas' College of Engineering and Technology, Kolkata 700 023, India.\\
\email{ushnishsarkar@yahoo.com}
}

\maketitle

\begin{abstract}
In this paper, we make some conjectures on prime numbers that are sharper 
than those found in the current literature. First we
describe our studies on Legendre's Conjecture which is still
unsolved. Next, we show that Brocard's Conjecture can be proved assuming 
our improved version of Legendre's Conjecture. Finally, we sharpen the 
Bertrand's Postulate for prime numbers. Our results are backed by extensive 
empirical investigation.
\end{abstract} 


\section{Introduction}
Investigation on the properties of prime numbers is an interesting area of 
study for centuries. Starting from Euclid's result~\cite{euclid} on the 
infinititude of primes to modern primality testing algorithms~\cite{kranakis},
research on prime numbers has grown exponentially. There exist many 
results on the set of prime numbers. Some of them has been proved, such as 
Bertrand's Postulate~\cite{nagell}. Some of them has remained conjectures till 
today, and amongst them, some important ones are: Goldbach's 
Conjecture~\cite{guy}, Legendre's Conjecture~\cite{hardy}, Brocard's 
Conjecture~\cite{weisstein} etc. We concentrate on sharpening Legendre's 
Conjecture and Bertrand's Postulate in this paper. Both of these are related 
to the number of prime numbers in certain intervals.

In this paper, we propose four novel conjectures on prime numbers.
\begin{enumerate}
\item Our analysis of Legendre's Conjecture brings forth two results.
\begin{enumerate}
\item In Conjecture~\ref{improvelegend}, we sharpen Legendre's Conjecture on 
the lower bound on the number of primes in $[n^2, (n+1)^2]$.
\item In Proposition~\ref{brocard}, we prove Brocard's Conjecture using our 
Conjecture~\ref{improvelegend}.
\end{enumerate}
\item Our Conjecture~\ref{legendbound} improves the upper bound on the number 
of primes in $[n^2, (n+1)^2]$ as given by Rosser and Schoenfeld's Theorem.
\item Our third conjecture and its consequences are two-fold.
\begin{enumerate} 
\item In Conjecture~\ref{improvebertrand}, we generalize Bertrand's Postulate 
on the number of primes in $[n,kn]$ (Bertrand's Postulate deals with $k = 2$ 
only).
\item As an implication of the above generalization, in 
Corollary~\ref{ub_nth_prime}, we propose a stronger upper bound on 
$n$-th prime than what follows from Bertrand's Postulate. 
\end{enumerate}
\item Our fourth and final conjecture proposes an upper bound on the number of 
primes in $[n,kn]$. To our knowledge, this problem has not been attempted 
before this work.
\end{enumerate}

\section{Sharpening Legendre's Conjecture}

Legendre's Conjecture~\cite{hardy} is an important unsolved problem in the 
domain of prime numbers. It states that:

\begin{proposition}
\label{legend}
[Legendre's Conjecture]\\
For every positive integer $n$, there exists
at least one prime $p$ such that $n^{2}< p < (n+1)^{2}$.
\end{proposition}

Although no formal proof has ever been found for this theorem,
empirical works suggest that the number of primes in such intervals
is not one but more. Let the number of primes between $n^{2}$ and
$(n+1)^{2}$ be called $leg(n)$. The following table contains brief results.

\begin{table}[htb]
\begin{center}
\begin{tabular}{|c|r|r|r|r|r|r|r|r|r|r|}
\hline
$n$        & 1   & 2   & 3   & 4   & 5   & 6   & 7   & 8   & 9   & 10 \\
\hline
$leg(n)$   & 2   & 2   & 2   & 3   & 2   & 4   & 3   & 4   & 3   & 5 \\
\hline
\end{tabular}
\end{center}
\caption{No. of primes between $n^2$ and $(n+1)^2$.}
\label{tab1}
\end{table}

Our empirical observations show that the minimum value of $leg(n)$ is always 
2 or more. It is never 1, as mentioned in Legendre's Conjecture. 
We have checked this for many random large values of $n$ and find that $leg(n)$ 
oscillates aperiodically, but shows a general upward trend as $n$ increases. 
Thus, it is extremely unlikely that it would ever be 1. So, we strengthen
Legendre's Conjecture as follows.

\begin{conjecture}
\label{improvelegend}
[Our Improved version of Legendre's Conjecture]\\
For every positive integer $n$ there exists at least 2 primes $p$
and $q$ such that $n^{2}< p < q < (n+1)^{2}$.
\end{conjecture}

We can derive a formula for the upper bound of $leg(n)$
from a known result, namely, Rosser and Schoenfeld's Theorem~\cite{rosser}.

\begin{proposition}
\label{rosser}
[Rosser and Schoenfeld's Theorem]\\
For any positive integer $n (> 17)$, 
if $\Pi(n)$ is the number of primes less than or
equal to $n$, then $\frac{n}{ln(n)}\leq \Pi(n) \leq \frac{1.25n}{ln(n)}$.
\end{proposition}

\begin{corollary}
\label{ub_legn}
[Upper Bound on $leg(n)$]\\
$leg(n)\leq \frac{n^{2}+10n+5}{8ln(n)}$.
\end{corollary}
\begin{proof}
From Theorem~\ref{rosser}, we know that 
$\Pi((n+1)^{2})\leq \frac{5(n+1)^{2}}{8ln(n)}$ and
$\Pi(n^{2})\geq \frac{n^{2}}{2ln(n)}$.
Subtracting, we have $leg(n)\leq \frac{n^{2}+10n+5}{8ln(n)}$. \qed
\end{proof}

Thus, we have an upper bound on $leg(n)$. However, we observe that
this theoretical upper limit is extremely loose. This is illustrated
in Table~\ref{tab2}. We here propose a tighter bound. By plotting 
the values of $leg(n)$ against consecutive integers $n$, we see a
general upward tendency in the curve. Studying the general nature of
the curve, we suggest the following conjecture.

\begin{conjecture}
\label{legendbound}
[Our Improved Bound on $leg(n)$]\\
For any positive integer $n$,
$\frac{n^{2}+10n+5}{3nln(n)} \leq leg(n) \leq\frac{n^{2}+10n+5}{3n}$.
\end{conjecture}

The data in Table~\ref{tab2} supports this conjecture.

\begin{table}[htb]
\begin{center}
\begin{tabular}{|r|r|r|r|r|}
\hline
$n$ & $leg(n)$ & $\substack{\mbox{Upper Bound}\\ 
\mbox{from Corollary~\ref{ub_legn}}}$ & \multicolumn{2}{r|}{$\substack{
\mbox{Our Bounds}\\ \mbox{from Conjecture~\ref{legendbound}:} \\
\mbox{Lower\ \ \ \ \ \ \ \ Upper}}$} \\
\hline
10 & 5 & 11.1 & 3.0 & 6.8 \\
\hline
20 & 7 & 25.2 & 3.4 & 10.1 \\
\hline
50 & 11 & 96.0 & 5.1 & 20.0 \\
\hline
100 & 23 & 298.7 & 8.0 & 36.7 \\
\hline
500 & 71 & 5129.1 & 27.3 & 170.0\\
\hline
1000 & 152 & 18276.6 & 48.7 & 336.7 \\
\hline
2000 & 267 & 66110.7 & 88.2 & 670.0 \\
\hline
5000 & 613 & 367638.8 & 196.1 & 1670.0 \\
\hline
20000 & 2020 & 5051250.9 & 673.5 & 6670.0 \\
\hline
45000 & 4218 & 23629958.8 & 1400.3 & 15003.3 \\
\hline
\end{tabular}
\end{center}
\caption{Comparing the bounds on $leg(n)$ from Rosser and Schoenfeld's Theorem
and our Conjecture~\ref{legendbound}.}
\label{tab2}
\end{table}

Moreover, we can prove Brocard's Conjecture~\cite{weisstein} using our 
Conjecture~\ref{improvelegend}.

\begin{proposition}
\label{brocard}
[Brocard's Conjecture]\\
Between $p_n^2$ and $p_{n+1}^2$, there exist at least 4 prime numbers, 
where $p_{n}$ is the $n$-th prime number.
\end{proposition}
\begin{proof}
From our Conjecture~\ref{improvelegend}, we have at least
2 primes between $p_n^2$ and $(p_n+1)^2$, and also 2
primes between $(p_{n+1}-1)^2$ and $p_{n+1}^2$. As the minimum
prime gap is 2, we have $p_{n+1}-p_n\geq 2$ Hence, we have
$p_{n+1}-1 \geq p_n+1$. And hence, $(p_{n+1}-1)^2 \geq
(p_n+1)^2$. So, the two intervals mentioned above are disjoint.
Hence, the number of primes between $p_n^2$ and
$p_{n+1}^2$ is at least 4. \qed
\end{proof}

\section{Generalization of Bertrand's Postulate}

Bertrand's Postulate~\cite{nagell} is an important theorem about the 
distribution of prime numbers. 

\begin{proposition}
\label{bertrand}
[Bertrand's Postulate]\\
For every positive integer $n>1$ there exists at least one prime
$p$ such that $n \leq p < 2n$.
\end{proposition}

In~\cite{nagura} it has been shown that for every $n>25$, there exists at
least 1 prime $p$ such that $n\leq p \leq \frac{6n}{5}$.

Inspired by this result, we investigated whether there exists 
at least $k-1$ primes between $n$ and $kn$ for any integer $n$.
This is a natural generalization of Bertrand's Postulate. We experimented 
with $k$ from 2 up to 1000000, and found that for any integer $n \geq a$,
\begin{enumerate}
\item there are at least $k-1$ primes between $n$ and $kn$ and 
\item the number of primes in such intervals keep on increasing almost 
monotonically as $n$ increases.
\end{enumerate}
The number $a$ is a threshold that slowly
increases as $n$ increases. The empirical results suggest that the variation in
$a$ with $k$ is given by 
$$a = \lceil 1.1\ln(2.5k) \rceil.$$

The variation of the threshold value $a$ is shown in Table~\ref{tab3} below 
for different values of $k$. The tabulated values of $k$  are those at
which the value of $a$ changes.

\begin{table}[htb]
\begin{center}
\begin{tabular}{|r|r|r|r|r|r|r|r|r|r|r|}
\hline
$k$ & 2 & 5 & 22 & 65 & 160 & 427 & 1020 & 200000 & 1000000 \\
\hline
$\substack{\mbox{Actual}\\ \mbox{Threshold}}$ & 2 & 3 & 5 & 6 & 7 & 8 & 9 & 14 & 16\\
\hline
$1.1\ln(2.5k)$ & 1.77 & 2.21 & 4.40 & 5.59 & 6.27 & 7.67 & 8.62 & 14.43 &
16.20 \\
\hline 
$\substack{\mbox{Our}\\ \mbox{Estimate}\\a}$ & 2 & 3 & 5 & 6 & 7 & 8 & 9 & 15 & 17\\
\hline
\end{tabular}
\end{center}
\caption{The threshold $a$ for different values of $k$.}
\label{tab3}
\end{table}

Based on the above results, we formally state our conjecture as below:
\begin{conjecture}
\label{improvebertrand}
[Our Generalization of Bertrand's Postulate]
For any integers $n$, $a$ and $k$, where $a = \lceil 1.1\ln(2.5k) \rceil$,
there are at least $k-1$ primes between $n$ and $kn$ when $n \geq a$.
\end{conjecture}
Here we have shown the results for certain selected values of $a$ and $k$ just
as illustration, however, but we have covered all intermediate
values in our experiments.

We also observe that, the actual number of primes in the interval
$[n,kn]$ actually increases far beyond $k-1$ as $n$ increases.
When $k=2$ (Bertrand's Postulate) it can be derived~\cite{apostol} from Prime
Number Theorem that the number of primes in the said interval is
roughly $\frac{n}{ln(n)}$, within some error limits.

The known upper bound for the $n$-th prime $p_n$ is $2^n$, i.e. $p_n < 2^n$.
This follows from Proposition~\ref{bertrand}. Using 
Conjecture~\ref{improvebertrand}, we can provide a stronger upper bound.
\begin{corollary}
\label{ub_nth_prime}
[Our Upper Bound on $n$-th Prime]\\
$p_n < 2^a (n-a)$ for any positive integer $a$, and the tightest bound is 
given by $a = \alpha+1$,
where $\alpha$ is the least positive integral solution of the inequality
$2^x > 1.1 \ln (2.5(n-x))$.
\label{ubound_nth_prime}
\end{corollary}
\begin{proof}
We have $p_n < p_a(n-a)$, since the interval $[p_a, p_n]$ contains 
$(n-a)$ primes. Thus, $p_n < 2^a (n-a)$. We obtain the tightest bound as 
follows. Using Conjecture~\ref{improvebertrand}, we can say that there are at 
least $n-a-1$ primes in the interval $[p_a, (n-a)p_a]$, when
$p_a > \lceil 1.1 \ln (2.5(n-a)) \rceil$. Hence, the best $a$ will satisfy
$2^a > \lceil 1.1 \ln (2.5(n-a)) \rceil$. \qed
\end{proof}

In table~\ref{tab4}, we compare the upper bounds on $n$-th prime as given by
Bertrand's Postulate and our Corollary~\ref{ub_nth_prime}.

\begin{table}[htb]
\begin{center}
\begin{tabular}{|r|r|r|r|}
\hline
$n$ & $p_n$ & Upper Bound $2^n$  & Our Upper Bound \\
\hline
32 & 131 & 4294967296 & 448 \\
\hline
987 & 7793 & 13072 \ldots (a 298-digit number) & 31424 \\
\hline
2000 & 17389 & 1148 \ldots (a 603-digit number) & 63840 \\
\hline
\end{tabular}
\end{center}
\caption{Comparing the upper bounds on $n$-th prime.}
\label{tab4}
\end{table}

We studied the number of primes in the intervals $[n, kn]$ for different
values of $n$ and $k$. We found from our observations based on plotting and 
curve-fitting that, for any particular $k$, the number of primes in the said 
interval increases almost linearly with $n$. Equations of the curves so 
obtained are roughly of the form $y=\frac{kn}{a}+bk$ where $y$ is the 
required number of primes, $a$ is close to 10 and $b$ is a number between 
1 and 10. $a$ and $b$, however, vary for different values of $n$ and $k$. By
manipulating this form, we can suggest an upper bound on the number
of primes in the interval $[n,kn]$, as follows:
\begin{conjecture}
\label{ub_numprimes}
[Our Upper Bound on the No. of Primes between $n$ and $kn$]
Given a positive integer $k$, the number of primes between $n$ and $kn$, for
any positive integer $n$, is bounded by $\frac{kn}{9}+k^2$.
\end{conjecture}
We defend this conjecture by briefly showing our results in Table~\ref{tab5}.
Each row corresponds to a single value of $n$ and each column corresponds to 
an individual value of $k$. Each entry in the matrix represents the 
number of primes between $n$ and $kn$ for the selected values of
$n$ and $k$. The value on top denotes the actual number of primes and
the value at bottom gives the upper bound on the number of primes as given by 
Conjecture~\ref{ub_numprimes}.

\begin{table}[htb]
\begin{center}
\begin{tabular}{|r||r|r|r|r|r|}
\hline
$n$ \ / \ $k$ & 2       & 5        & 10       & 50        & 100\\
\hline
\hline
\multirow{2}{*}{10}
& 4 & 11 & 21 & 91 & 164\\
& 6.2 & 30.6 & 111.1 & 2555.6 & 10111.1 \\
\hline
\multirow{2}{*}{50}
& 10 & 38 & 80 & 352 & 654 \\
& 15.1 & 52.8 & 155.6 & 2777.8 & 10555.6\\
\hline
\multirow{2}{*}{100}
& 21 & 70 & 143 & 644 & 1204 \\
& 26.2 & 80.6 & 211.1 & 3055.6 & 11111.1\\
\hline
\multirow{2}{*}{500}
& 73 & 272 & 574 & 2667 & 5038\\
& 115.1 & 302.8 & 655.6 & 5277.8 & 15555.6\\
\hline
\multirow{2}{*}{1000}
& 135 & 501 & 1061 & 4965 & 9424 \\
& 226.2 & 580.6 & 1211.1 & 8055.6 & 21111.1\\
\hline
\multirow{2}{*}{5000}
& 560 & 2094 & 4464 & 21375 & 40869 \\
& 1115.1 & 2802.8 & 5655.6 & 30277.8 & 65555.6 \\
\hline
\end{tabular}
\end{center}
\caption{Comparing actual no. of primes between $n$ and $kn$ with that 
predicted by Our Conjecture~\ref{ub_numprimes}.}
\label{tab5}
\end{table}

We have experimented with the intermediate values of $n$ and $k$ as
well, and found the conjecture to hold. But due to lack of space we
do not show the results here. However, we find that this bound is
quite tight only for small values $k$. As $k$ increases, the upper
bound becomes increasingly loose. However, no violation has been
detected till $k=1000000$, and with the upper bound increasing
monotonically, it is unlikely that any violation will ever occur.
Also, between $n$ and $kn$ there are $kn-n+1$ numbers. For small
values of $n$ and $k$, the upper bound of $\frac{kn}{9}+k^{2}$ may
be greater than this value, in which our conjecture obviously does
not provide any new information. We find that, the suggested upper
bound is less than $kn-n+1$ only for $n \geq \frac{9k^{2}-9}{8k-9}$.
Analyzing this graphically, we find that this approximately requires
$n \geq 2k$.

\section{Conclusion}
This paper attempts at sharpening Legendre's Conjecture and generalizing 
Bertrand's Postulate on prime numbers. We present both formal arguments
and empirical supports to defend our new conjectures and their corollaries.
Further investigation is required for attempting to prove these results.

\end{document}